\documentclass[11pt,reqno]{amsart}
\usepackage{mathrsfs, amsmath,amssymb,amsfonts, amsthm, dsfont}
\usepackage{float}
\restylefloat{table}
\usepackage{hyperref}
\usepackage{epsfig}
\usepackage{amscd}
\usepackage{amsmath, amssymb, amsthm}
\usepackage{graphics}
\usepackage{color}
\usepackage{dsfont}

\newtheorem{theorem}{Theorem}[section]
\newtheorem{conjecture}{Conjecture}[section]
\newtheorem{proposition}[theorem]{Proposition}
\newtheorem{corollary}[theorem]{Corollary}
\newtheorem{lemma}[theorem]{Lemma}
\newtheorem{remark}[theorem]{Remark}
\newtheorem{definition}[theorem]{Definition}

\numberwithin{equation}{section}

\renewcommand{\Pi}{\varPi}

\renewcommand{\epsilon}{\varepsilon}

\numberwithin{equation}{section}

\newenvironment{rem}{\medskip\noindent{\it Remark:\/} }{\medskip}

\newcommand{\R}{{\mathbb R}}

\newtheorem*{main-theorem}{Main Theorem}
\newtheorem*{old-thm}{Theorem}
\theoremstyle{definition}
\numberwithin{equation}{section}

\def\11{\mathds{1}}

\def\phi{\varphi}

\def\be{\begin{eqnarray*}}
\def\ee{\end{eqnarray*}}
\def\ben{\begin{eqnarray}}
\def\een{\end{eqnarray}}

\def\L2R{L_{\text{Rest}}^2}

\newcommand{\lcal}{\mathcal{L}}

\newcommand{\ocal}{\mathcal{O}}

\begin{document}
\title[Number of nodal domains]{Number of nodal domains of eigenfunctions on non-positively curved
surfaces with concave boundary}

\author{Junehyuk Jung}
\address{Department of Mathematical Science, KAIST, Daejeon 305-701, South Korea}
\curraddr{School of Mathematics, IAS, Princeton, NJ 08540, USA}
\email{junehyuk@ias.edu}

\author{Steve Zelditch}
\address{Department of Mathematics, Northwestern  University, Evanston, IL 60208, USA}
\email{zelditch@math.northwestern.edu}

\begin{abstract}
It is an open problem in general to prove that there exists a sequence of $\Delta_g$-eigenfunctions $\phi_{j_k}$  on
a Riemannian manifold $(M, g)$ for which the number $N(\phi_{j_k}) $ of nodal domains tends to infinity with the
eigenvalue. Our main result is that $N(\phi_{j_k}) \to \infty$ along a subsequence of eigenvalues of density 1 if
the $(M, g)$ is a non-positively curved surface with concave boundary, i.e. a generalized Sinai or Lorentz billiard.
Unlike the recent closely related  work of Ghosh-Reznikov-Sarnak and of the authors on the nodal domain  counting problem,
the surfaces need not have any symmetries.

\end{abstract}
\maketitle

\section{Introduction}

Let $(M, g)$ be a surface with  non-empty smooth boundary $\partial M \neq \emptyset$.  We consider the eigenvalue problem,
$$\left\{ \begin{array}{l}  -\Delta \phi_{\lambda} = \lambda^2 \phi_{\lambda},\\
B \phi_{\lambda}  = 0 \;\; \mbox{on}\;\; \partial M \end{array} \right.,$$
where $B$ is the boundary operator,
e.g. $B \phi = \phi|_{\partial M}$ in the Dirichlet case or $B
\phi = \partial_{\nu} \phi|_{\partial M}$ in the Neumann case.     We  denote by $\{\phi_{j}\}$
an orthonormal basis of eigenfunctions,  $ \langle \phi_j, \phi_k \rangle = \delta_{jk}$, with
$0 = \lambda_0 < \lambda_1 \leq \lambda_2 \leq \cdots $  counted with multiplicity. The inner product is
defined by   $\langle f, g \rangle = \int_{M} f \bar{g} dA$ where $dA$ is the area form of $g$.
We denote the nodal line of $\phi_{\lambda} $ by
\[
Z_{\phi_{\lambda}} = \{x: \phi_{\lambda}(x) = 0\}.
\] We also denote by $N(\phi_{\lambda} )$ the number of nodal domains of $\phi_{\lambda}$, i.e. the number of connected components
$\Omega_j$ of
$$M \backslash (Z_{\phi_{\lambda}} \cup \partial M)=  \bigcup_{j = 1}^{N(\phi_{\lambda} )} \Omega_j. $$
 The connected components are called the  nodal domains of $\phi_{\lambda}$.
We further denote by
\[
\Sigma_{\phi_{\lambda}} = \{x \in Z_{\phi_{\lambda}}: d \phi_{\lambda}(x) = 0\}
\]
the  singular set of $\phi_{\lambda}.$
The main result of this article, developing the method of \cite{JZ},  gives a rather  general sufficient condition on surface $M$ with non-empty boundary $\partial M \not= \emptyset$ and ergodic billiard flow under which the number of nodal domains tends
to infinity along a full density subsequence (i..e of `almost all' eigenfunctions  of any orthonormal basis).  A significant gain in the billiard case is that we do not require
the surface (or eigenfunctions) to have a symmetry.

\begin{center}
\includegraphics[scale=.8]{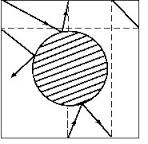}
 \includegraphics[scale=.6]{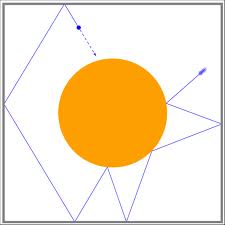}
\end{center}
\bigskip

We first state the result for the special case of  ``Sinai-Lorentz" billiards, i.e. non-positively curved surfaces with concave boundary.     We denote the scalar curvature of $(M, g)$
by $K$.

\begin{theorem}\label{theo1}
Let $(M, g) $ be a non-positively curved  surface $K \leq 0$ with non-empty smooth
concave boundary $\partial M$. Let  $\{\phi_j\}$ be an orthonormal eigenbasis of Dirichlet (resp. Neumann) eigenfunctions.
 Then there exists a subsequence $A \subset \mathbb{N}$
of density one so that
\[
\lim_{\substack{j \to \infty \\ j \in A}}N(\phi_j) = \infty.
\]

\end{theorem}

By a surface of non-positive curvature with concave boundary,  we mean
 a non-positively curved surface $(M, g)$,
\begin{equation} \label{M}
M= X \backslash \bigcup_{j = 1}^r \mathcal{O}_j,
\end{equation}
obtained by removing a finite union $\mathcal{O} : = \bigcup_{j = 1}^r \mathcal{O}_j$ of embedded nonintersecting geodesically  convex domains (or `obstacles')
$\mathcal{O}_j$ from a closed non-positively curved surface $(X, g)$ without
boundary. It   is proved in \cite{Si1,cs,bcs} that the billiard flow of a Sinai-Lorentz billiard is ergodic.

We give in fact a  more general condition,  which requires
some further notation and terminology.

\bigskip

\begin{definition} \label{CDDEF}
The Cauchy data of Neumann (resp. Dirichlet) eigenfunctions of $(M,g)$ is defined by
$$\left\{ \begin{array}{ll}  \phi_j^b = \phi_j |_{\partial M}, & \; \mbox{ Neumann boundary conditions}, \\ \\
\phi_j^b =
\lambda_j^{-1} \partial_{\nu} \phi_j |_{\partial M}, & \; \mbox{Dirichlet boundary conditions}, \end{array} \right.
$$
\end{definition}

The more general result is:

\begin{theorem}\label{theo1a}
Let $(M, g) $ be a surface with non-empty smooth boundary $\partial M$. Let  $\{\phi_j\}$ be an orthonormal eigenbasis of Dirichlet (resp. Neumann) eigenfunctions. Assume that $(M, g)$ satisfies
 the following conditions:

\begin{itemize}

\item[(i)]  The billiard flow $G^t$ is ergodic on $S^* M$ with respect to Liouville measure;
\bigskip


\item[(ii)]  The sup norms of the Cauchy data
\[
(\phi_j |_{\partial M}, \lambda_j^{-1} \partial_{\nu} \phi_j |_{\partial M})
\]
are  $o(\lambda_j^{\frac{1}{2}})$ as $\lambda_j \to \infty$.

\end{itemize}

   Then there exists a subsequence $A \subset \mathbb{N}$
of density one so that
\[
\lim_{\substack{j \to \infty \\ j \in A}}N(\phi_j) = \infty.
\]

\end{theorem}

Theorem \ref{theo1} follows from Theorem \ref{theo1a} combined with
some new results of the second author with C. Sogge.  The
ergodicity condition
(i) is known to be satisfied for a non-positively  curved surface with
concave boundary  \cite{KSS} (see \S \ref{BILLIARDS}).   Moreover, in   \cite{SoZ3} the second condition (ii) is proved for such surfaces,
among many other cases. Using the Melrose-Taylor diffractive parametrix
on manifolds with concave boundary, the following is proved:

\begin{theorem} \label{SoZ} \cite{SoZ3}
Let $(M, g)$ be a Riemannian manifold of dimension n with geodesically concave boundary. Suppose that
there exist no boundary  self-focal points $q \in \partial M$.  Then
the sup-norms of Cauchy data of Dirichlet, resp. Neumann,  eigenfunctions are  $o(\lambda_j^{\frac{n-1}{2}})$.
\end{theorem}

The term {\it boundary self-focal point}  is a
dynamical condition on   the billiard flow
 $\Phi^t$ of  $(M, g, \partial M)$, i.e. the geodesic flow in the interior with
elastic reflection on $\partial M$.  See \S \ref{BILLIARDS} for background. We  denote the broken
exponential map by
$\exp_x \xi = \pi \Phi^1(x, \xi)$ (see \cite{HoI-IV} (Chapter XXIV)  or \cite{MT} for background.)
 Given any $x\in \bar{M}$, we denote by  $\lcal_x$ the set
of loop directions at $x$,
\begin{equation} \lcal_x = \{\xi \in S^*_x M : \exists T: \exp_x T
\xi = x \}.
\end{equation}

\begin{definition}
We say that   $q$ is a boundary self-focal  point if $|\lcal_q| > 0$ where
$|\cdot|_q$ denotes the surface measure on   $S^*_{q, \rm{in}} M$, the set of inward pointing unit tangent
vectors  determined by
the metric $g_q$. Equivalently, the measure of the set of $\xi \in B^*_q \partial M$ of
the co-ball bundle $B^*_q \partial M$ of the boundary at $q$ for which
$\pi \beta (q, \xi) = (q, \eta)$ for some $\eta \in B^*_q \partial M $ has measure zero. Here, $\beta$ is the billiard map \eqref{BETA}.
\end{definition}

Billiards on
non-positively curved billiards  on  $(M, g)$ with concave boundary  never
have self-focal points.  Self-focal points $q$ are necessarily self-conjugate, i.e. there
exists a broken  Jacobi field along a geodesic billiard loop at $q$ vanishing at both endpoints. But as we review
in \S \ref{BILLIARDS}, non-positively curved dispersive billiards
as above do not have conjugate points.
Thus, Theorem \ref{SoZ} holds and so all of the hypotheses are Theorem \ref{theo1a} are satisfied by
the surfaces of Theorem \ref{theo1}.

Theorem \ref{SoZ} is the analogue for Cauchy data of the sup norm results on
manifolds without boundary  proved in \cite{SoZ,SoZ2}.  The proof is too
lengthy to be included here and will be published elsewhere.

\subsection{Outline of the proof of Theorem \ref{theo1a}}

We plan to deduce Theorem \ref{theo1a} from Theorem \ref{SoZ} and from
a variety of additional analytical and topological arguments. The principal
analytical results, besides Theorem \cite{SoZ}, are the quantum ergodic
restriction theorems of \cite{HZ,ctz}.
The overall argument follows the strategy of  \cite{JZ}, and is based on counting the zeros of  the Cauchy data  $\phi_j^b$  of an ergodic sequence of eigenfunctions on $\partial M$ and relating them to numbers of nodal domains
using the Euler inequality for embedded graphs.

We now sketch the proof of Theorem \ref{theo1a}.
The analytical part of the proof is the following result about boundary nodal points:

\begin{theorem}\label{theoS}

Let $(M, g)$  be a Riemannian surface which satisfies the assumptions of Theorem \ref{theo1a}.  Then for any given orthonormal eigenbasis of Neumann eigenfunctions $\{\phi_j\}$, there exists a subsequence $A \subset \mathbb{N}$ of density one such that
\[
\lim_{\substack{j \to \infty \\ j \in A}} \#  Z_{\phi_j} \cap \partial M = \infty.
\]
Furthermore, there are an infinite number of zeros where  $\phi_j |_{\partial M }$ changes sign. For Dirichlet eigenbasis $\{\phi_j\}$, there exists a subsequence $A \subset \mathbb{N}$ of density one such that
\[
\lim_{\substack{j \to \infty \\ j \in A}} \#  \Sigma_{\phi_j} \cap \partial M = \infty.
\]
\end{theorem}

Given Theorem \ref{theoS}, the remainder of the proof of Theorem \ref{theo1} is topological. As in \cite{JZ}, we   use an Euler characteristic argument for
embedded graphs, Theorem \ref{topology}, to obtain a lower bound on the number of nodal domains
from the lower bound on the number of boundary zeros.
  Heuristically, the nodal lines which touch the
boundary transversally must intersect at two points of the boundary and
trap nodal domains. This is not literally correct when the genus is $ > 0$
but the genus correction is bounded and does not affect the growth rate
of the number of nodal domains. The role of the boundary in \cite{GRS}
and \cite{JZ} was played by the fixed point set of an anti-holomorphic
involution. No  symmetry assumption is necessary in the boundary case, because a nodal segment which has an end point at $\partial M$ must terminate in $\partial M$.

\subsection{Outline of the proof of Theorem \ref{theoS}}

The key step in proving Theorem \ref{theoS} is the following

\begin{proposition} \label{PROP}

Let $(M, g) $ be a Riemannian surface with smooth boundary
and   ergodic billiards.
Then  exists a subsequence of density one of the Neumann eigenfunctions such that, for any fixed arc (or interval) $\beta \subset \partial  M$,
\[
\int_{\beta} |\phi_j | ds >  \left|\int_{\beta} \phi_j ds\right|.
\]
Moreover, there exists a density one subsequence of Dirichlet eigenfunctions such that
for any fixed arc $\beta \subset \partial M $,
\[
\int_{\beta} |\partial_{\nu} \phi_j | ds >  \left|\int_{\beta} \partial_{\nu} \phi_j ds\right|.
\]
\end{proposition}
That is,
\begin{equation} \label{bnotation}
\int_{\beta} |\phi_j^b | ds >  \left|\int_{\beta} \phi_j^b ds\right|.
\end{equation}

The Proposition clearly implies that the Neumann eigenfunctions $\phi_{\lambda}$ must have a sign-changing zero on any arc $\beta$ for a density one subsequence
of eigenfunctions. Similarly for normal derivatives of Dirichlet eigenfunctions. Theorem
\ref{theoS} is a direct consequence of Proposition \ref{PROP}. By an
`arc' or interval of $\partial M$ we simply mean the image of an interval
under a parametrization.

To prove Proposition \ref{PROP} we combine three results, two of which are
proved elsewhere  and one which we prove here. The first is the quantum ergodic restriction theorem
\eqref{QES} discussed above.
The second ingredient is  the following ``Kuznecov sum formula", extending the result of \cite{z} to manifolds with boundary.  It is an immediate consequence of Theorem 1 of \cite{HHHZ}:

\begin{theorem}[\cite{HHHZ}]\label{K}

Let $(M, g)$ be a Riemannian surface with smooth
boundary $\partial M$. Let $\{\phi_j\}$ be an orthonormal basis of Neumann eigenfunctions.  For any given fixed $f \in C_0^{\infty}(\partial M)$,
\[
\sum_{\lambda_j < \lambda}\left|\int_{\partial M} f \phi_j ds\right|^2 = \left(\frac{2}{\pi} \int_{\partial M} f^2 ds\right) \lambda + o(\lambda).
\]

 Let $\{\phi_j\}$ be an orthonormal basis of Dirichlet eigenfunctions.  For any given fixed $f \in C_0^{\infty}(\partial M)$,
\[
\sum_{\lambda_j < \lambda}\left|\lambda_j^{-1} \int_{\partial M} f \partial_{\nu} \phi_j ds\right|^2 =\left(\frac{2}{\pi} \int_{\partial M} f^2 ds\right)\lambda + o(\lambda).
\]
\end{theorem}

We prove Theorem \ref{K} in \S \ref{KUZSECT}.  We view Theorem \ref{K} as an asymptotic mean formula for a sequence
of probability measures, stating that on average, $ \left|\int_{\partial M} f \phi_j^b ds\right|^2$ is of order
$\lambda_j^{-1}$.
An application of Chebychev's inequality then gives
\begin{corollary}\label{cor81}
There exists a constant $c=c_f>0$ depending only on $f$ such that, for each $M > 0$ there exists a subsequence of proportion $\geq 1 - \frac{c}{M}$ of the $\{\phi_j\}$ for which
\[
 \left|\int_{\partial M} f \phi_j^b ds  \right| \leq M \lambda_j^{-\frac{1}{2}}
\]
\end{corollary}


The next ingredient is the quantum ergodicity of  Cauchy data of ergodic sequences of eigenfunctions along the boundary.
 In \cite{HZ,ctz} (see also
\cite{Bu}) it is proved that if the geodesic billiard flow of a Riemannian manifold $(M, g)$ is
ergodic, then there exists a subsequence $A\subset \mathbb{N}$ of density one so that,
for any $f \in C(M)$,  \begin{equation}\label{QES} \begin{gathered}
\lim_{\substack{j \to \infty\\ j \in A}} \int_{\partial M} f |\phi_j^b|^2 d s
= \omega_B(f),
\end{gathered}
\end{equation}
where $\omega_B(f)$ is a `limit state' (i.e. a positive measure viewed as a linear functional on $C(B^* \partial M)$)
which depends on the boundary condition B. The limit state  is defined in \S \ref{QER} and the result is stated in Theorem \ref{main}.
We refer to the density one sequence \eqref{QES} as a ``boundary ergodic sequence"  of eigenfunctions.

We now put together   Theorem  \ref{SoZ},
 Theorem \ref{K}, and \eqref{QES} to given an outline of the proof Proposition \ref{PROP}.
For simplicity we assume that the boundary conditions are Neumann.
From the third assumption in Theorem \ref{theo1}, we have
$$||\phi_{j} |_{\partial M}||_{L^{\infty}}\leq \lambda_j^{1/2} o(1). $$

We then prove that for any $M > 0$ there exists a subsequence of density $\geq 1 - \frac{1}{M}$ so that for
any arc $\beta \subset H$,
$$|\int_{\beta} \phi_j ds| \leq C \lambda_j^{-1/2}  M$$
and there exists a subsequence of density one for which
$$\int_{\beta} |\phi_j | ds \geq ||\phi_j ||_{C^0(\beta)}^{-1} ||\phi_j ||_{L^2(\beta)}^2
\geq C \lambda_j^{-1/2} \frac{1}{o(1)}, $$
This is a contradiction if $\phi_j$ has no sign change on $\beta $.

It follows that for any $M > 0$ there exists a subsequence of density $\geq 1 - \frac{1}{M}$ for which
$$\int_{\beta} |\phi_j | ds > |\int_{\beta} \phi_j ds|. $$
This implies the existence of a subsequence of density one with this property.

\subsection{Related and future work}

For the sake of completeness, we  recall that in \cite{JZ}  it is  proved  that $N(\phi_{\lambda})$
tends to infinity along a density one sequence of even or odd eigenfunctions  on
a  Riemannian surfaces $(M, g, \sigma)$ of negative curvature with an isometric orientation reversing involution $\sigma$ with separating fixed point set
 $\rm{Fix}(\sigma)$.   In \cite{JZ}, the assumption of the existence of $\sigma$ is necessary in order to ensure that there are at most two intersections between $\rm{Fix}(\sigma)$ and a closed nodal curve, which allows one to relate the number of nodal domains and number of intersections.

We also  conjecture that Theorem \ref{theo1a}  generalizes to all hyperbolic billiards, such as
planar billiard tables with concave walls and corners, and also  to the Bunimivoch
stadium (with a continuous tangent line). The only part of the proof
which is not contained in this article is the validity of
the conclusion of Theorem  \ref{SoZ}  on such manifolds with corners.
The corners are serious
complications   to the proof of Theorem \ref{SoZ} since sequences of eigenfunctions
might be exceptionally large there.  However for the
purposes of this article,    sup norm estimates
on these manifolds are only needed away from the corners, and it may be possible
to prove the necessary bounds without proving all of Theorem \ref{SoZ}
in the cornered case.

We also conjecture that Theorem  \ref{SoZ} should be true for general Riemannian manifolds with smooth boundary with
no boundary self-focal points. This is work in progress of the second author with C. Sogge.
If we can generalize Theorem \ref{SoZ} to any smooth boundary with
no boundary self-focal points,
then the condition (ii) in Theorem \ref{theo1a} can be replaced by
the purely dynamical condition,  [(ii)'] There does not exist a boundary {\it self-focal} point $q \in \partial M$ for the billiard flow, and one would have
 purely dynamical conditions (i)-(ii)' ensuring growth of numbers of nodal
domains (as in the special case of  Theorem \ref{theo1}). One may further
ask if ergodicity itself prohibits existence of self-focal points on the boundary,
at least in the real analytic case. As observed in \cite{SoZ}, real analytic surfaces without
boundary and ergodic geodesic flow cannot have self-focal points $p$ because
the flowout of $S^*_pM$ under the geodesic flow is an invariant Lagrangian
torus enclosing a positive measure invariant set. But in the boundary case,
the flowout of the inward pointing $S^*_{p, \rm{in}} M$ at $p \in \partial M$
might not bound an invariant set.

The intersection points $
Z_{\phi_{\lambda}} \cap \partial M$ in the Neumann case are the points where the nodal  set touches the boundary in the sense of \cite{TZ2}.
That article shows that if $\partial M$ is  piecewise real analytic, then the number of intersection points is $\leq C_{M} \lambda$ for some constant
depending only on the domain. In theorem \ref{theoS},  we only show that
the number of intersection points tends to infinity in our setting. It would
be very interesting to have a quantitative lower bound.
For  Sinai-Lorentz billiards we conjecture that the sup norms of
the Cauchy data  are of order  $O(\frac{\lambda_j^{\frac{n-1}{2}}}{\sqrt{\log \lambda_j}})$. This would be a key step in producing zeros in intervals
of the boundary of lengths $\frac{1}{\sqrt{\log \lambda_j}}$, thereby
producing  logarithmic lower bounds on numbers of nodal domains.

\subsection{Acknowledgements}

This paper makes use of recent joint work of the second author with several
collaborators, which in part was motivated by the applications to nodal sets. Theorem \ref{K} is recent joint work with X. Han, A. Hassell and H. Hezari \cite{HHHZ}. It also
uses calculations in recent  work  \cite{HeZ} with H. Hezari.As mentioned above, Theorem \ref{SoZ} is joint work  with  C. D. Sogge
\cite{SoZ3}.   The
boundary quantum ergodicity theorem and boundary local Weyl law is joint work with A. Hassell \cite{HZ} and  with H. Christianson and J. Toth \cite{ctz}.  We would also like to thank N. Simanyi for helpful correspondence
regarding \cite{KSS} and L. Stoyanov for correspondence on billiard problems.

The first author was supported by the National Research Foundation of Korea(NRF) grant funded by the Korea government(MSIP)(No. 2013042157). The first author was also partially supported by NSF grant DMS-1128155 and by TJ Park Post-doc Fellowship funded by POSCO TJ Park Foundation. Research of the second author was partially supported by NSF grant DMS-1206527.

\section{ Billiards on negatively curved surfaces exterior to convex obstacles }\label{BILLIARDS}

Before discussing billiards on these surfaces, we introduce some general notation and background.

\subsection{Billiard map}\label{BMAP}

We  denote by \begin{equation} \label{GT} \Phi^{t}:
S^{*} M
  \rightarrow S^*M\end{equation}  the billiard   (or broken geodesic)   flow
 on
  $S^*M$. The trajectory $\Phi^{t}(x, \xi)$ consists of geodesic  motion
  between impacts with the boundary, with  the usual reflection
  law at the boundary.

The   billiard map $\beta$
is defined on the unit ball bundle $B^*\partial M$ of $\partial M$  as follows: given $(y, \eta)
\in B^*
\partial  M$, i.e.  with $|\eta| < 1$, we let $(y, \zeta) \in S^*
M$ be the unique inward-pointing unit covector at $y$ which
projects to $(y, \eta)$ under the map $T^*_{\partial M}
\overline{M} \to T^* \partial M$.
 Then we follow the geodesic  \begin{equation} \label{EDEF} (y, \eta) \in B^* (\partial M) \to
    \Phi^t(q(y),  \zeta(y, \eta))  \end{equation} until the projected billiard orbit  first
intersects the boundary again; let $y' \in
\partial M $ denote this first intersection.
We denote  the inward unit normal vector at $y'$ by $\nu_{y'}$,
and  let $\zeta' = \zeta + 2 (\zeta \cdot \nu_{y'}) \nu_{y'}$ be the
direction of the geodesic after elastic reflection at $y'$. Also,
let $\eta'$ be the projection of $\eta'$ to $T^*_{y'}\partial
M$. Then  \begin{equation} \label{BETA}
\beta(y, \eta) := (y', \eta'). \end{equation}

The directions tangent to $\partial M$ cause singularities (discontinuities) in $\beta$, and
the billiard map is not apriori well-defined on initial
directions which are tangent to $\partial M$, i.e. $(y, \eta)
\in S^*
\partial M$ with $|\eta| = 1$.  To obtain a smooth dynamical system, one often  removes the tangential directions $S^* \partial M$.
 Billiard trajectories
starting on $\partial M$ in non-tangential directions may  become tangential
at some future intersection and one often punctures out all trajectories which in the past or future become
tangential.
That is, one defines $B_q^0 = B^*_q \partial M \backslash S^*_q \partial M$ to be the projections of non-tangential directions
and define
$\mathcal{R}^1 \subset B_q^0 $ to be $\beta^{-1}(B_q^0 )$ and, more generally, $\mathcal{R}^{k+1} = \beta^{-1}(\mathcal{R}^{k})$ for natural numbers $k$. Thus $\mathcal{R}^k$ consists of the points where $\beta^k$ is well defined and maps to $B_q^0 $.
Similarly one defines $\mathcal{R}^{-1} = \beta_-^{-1}B_q^0 $ and $\mathcal{R}^{-k-1} = \beta_-^{-1}(\mathcal{R}^{-k})$.
Clearly $\mathcal{R}^1 \supset \mathcal{R}^2 \supset \dots$, $\mathcal{R}^{-1} \supset \mathcal{R}^{-2} \supset \dots$ and it is shown in \cite{CFS} that each $\mathcal{R}^k$ has full measure. Then $\beta$ is a symplectic diffeomorphism
of  $$\mathcal{R}^\infty = \bigcap_k \mathcal{R}^k. $$   However,  in this
article we
define $\beta$ at tangent vectors to $\partial M$ so that the billiard trajectory is simply a geodesic of $(X,g)$
which hits $\partial M$ tangentially. The billiard map is then defined on all of $B^* \partial M$ but is discontinuous
along the set of tangential directions.

\subsection{Ergodicity of the billiard map for non-positively curved surfaces with concave boundary}

We need the following result of Kramli-Simonyi-Szasz \cite{KSS}.
\begin{proposition}[\cite{KSS}]\label{E}
Billiards on a  non-positively curved surface with concave boundary are ergodic.
\end{proposition}

Since the result is not stated this way in \cite{KSS},  we briefly review the proof.
  The main result of the \cite{KSS} is a proof of the
``Fundamental Theorem for Dispersing Billiards"   (Theorem 5.1).  The
``dispersing billiards"  condition implies the existence
almost everywhere of stable/unstable foliations for the billiard map and also
indicates some of its quantitative properties.
   In section 6 of
\cite{KSS} the authors show how ergodicity of the billiards follows from their Theorem 5.1 by the so-called Hopf-Sinai argument.

 The surfaces in \cite{KSS}  are assumed to be the exterior
\eqref{M} of a finite union of convex obstacles (i.e. the boundary curves have strictly positive geodesic curvature
from the inside). They  are more general than the  non-positively curved surfaces assumed  here.  The billiards
are  only assumed to satisfy `Vetier's conditions',  which are conditions implying that `no focal points arise'.  The conditions
are stated precisely in Condition 1.2-1.4 in \cite{KSS}.  Condition (1.2) is that the distance between obstacles
is bounded below by some $\tau_{\min} > 0$, which is obvious for a compact surface when the obstacles do not
intersect. Condition (1.3) is that there exists $\tau_{\max}$ so that any geodesic must intersect $\partial M$
in time $\leq \tau_{\max}$. This condition can be removed if the curvature is strictly negative.  Condition (1.4) is a curvature condition which is satisfied as long as $K \leq 0$.
In this case, Condition (1.3) becomes irrelevant. In fact, $K \leq 0$ along implies the fundamental theorem
and ergodicity of the billiard flow.

For $(x, \xi) \in S^*_x M$, the stable (resp. unstable) fiber  $H^{(s)}(x, \xi)$
(resp. $H^{(u)}(x, \xi)$)  through $x$ is the set of $(y, \eta) \in S^* M$ so that
\[
\lim_{t \to +\infty} d(\Phi^t(y, \eta), \Phi^t(x, \xi)) =0
\]
(resp. $t \to - \infty$).  The stable leaf through $x$ is $\bigcup_{t \in \mathbb{R}} \Phi^t(H^{(s)}(x)). $
Similarly for the unstable leaf.  Vetier proved that under the conditions above, there exist stable and unstable fibers
through almost $(x, \xi) \in S^* M$ which are $C^1$ curves. It follows that through almost every $(q, \eta) \in B^* \partial M$
there exist stable/unstable leaves for the billiard map, which are $C^1$ curves. In particular, this is the case
for non-positively curved surfaces with concave boundary.

\subsection{Absence of self-focal points  non-positively curved surfaces with concave boundary}
The follow Lemma lets us use Theorem \ref{SoZ} in Theorem \ref{theo1a}.
\begin{lemma} \label{NFP}
There do not exist partial self-focal points on a non-positively curved surfaces with concave boundary.
\end{lemma}

To prove this, we consider broken  Jacobi fields for the billiard flow and begin with some background from \cite{W,ZL,B}.
  A  (normal) Jacobi field
is an orthogonal vector field $J(t) $  along a billiard trajectory  $\gamma$ with transversal reflections at $\partial M$,
 which satisfies the Jacobi equation $\frac{D^2}{dt^2} J +
K(\gamma(t)) J = 0 $ away from the elastic impacts, and which is reflected by the law
$$\begin{pmatrix} J \\  J' \end{pmatrix} \to \begin{pmatrix} -1 & 0 \\ & \\
\frac{2 K(s)}{\sin \phi(s)} & -1 \end{pmatrix} \begin{pmatrix} J \\  J' \end{pmatrix} = \begin{pmatrix} - J \\
\frac{2 K(s)}{\sin \phi(s)} J - J' \end{pmatrix} $$
at the reflection point. Here $\phi$ is the angle that $\gamma'(t)$ makes with the boundary at the impact time.
Recalling that a Jacobi field is the variation vector field $J(t) = \frac{\partial}{\partial \epsilon} \gamma_{\epsilon}(t)$
of a 1-parameter family of billiard trajectories, we see that the reflection law is the derivative in $\epsilon$ of the
reflection law for the curves $\gamma_{\epsilon}$.

\begin{proof}
Assume first that $Y$ is a simply connected non-positively curved surface, and that $\ocal_1,\cdots,\ocal_m \subset Y$ are disjoint obstacles. Fix a point $q \in \partial \ocal_1$ and consider billiard trajectories of $(q, \eta)$ on $Y \backslash \bigcup_{j=1}^m \ocal_j$. For a given billiard trajectory of $(q,\eta)$, we correspond a sequence $\{a_j(\eta)\}_{j\geq 0}$ with $1 \leq a_j(\eta)\leq m$ such that $j$-th impact occurs on the boundary of $\ocal_{a_j}$ (we assume that $a_0(\eta)=1$.)

For a given sequence $B=\{b_j\}_{0\leq j\leq M}$ with $1 \leq b_j \leq m$, let
\[
S_B:=\{\eta ~|~ a_j(\eta)=b_j,~j=0,\cdots,M\}.
\]
Let $q_j(\eta)\in \partial \ocal_{a_j(\eta)}$ be the $j$-th impact point. Since $Y$ is simply connected, the length of billiard trajectory from $q$ to $q_M(\eta)$ is bounded from above by some constant for $\eta \in S_B$. Therefore, if there are infinitely many $\eta \in S_B$ such that $q_M(\eta)=q'$ with same $q'$, then $q$ and $q'$ are conjugate. However, this is impossible, since the norm of the  Jacobi field is monotonically increasing between impacts
and only grows at an impact. Hence along any billiard trajectory from any $q \in \partial M$, it cannot vanish for $t > 0$.

Now let $\tilde{X}$ be the universal covering of $X$ and let $\tilde{\ocal}_j$ for $j=1,2,\cdots$ be the lift of obstacles on $X$. Fix $q \in \partial M$. For each $(q,\eta) \in S^*M$ whose billiard trajectory is a loop, we correspond a finite length sequence $\{a_j(\eta)\}_{j\leq T}$ as above, where $T>0$ is the first index such that $q_T(\eta)$ is a preimage of $q$ on $\tilde{X}$. Then from above, we infer that there are finitely many $\eta$ such that $\{a_j(\eta)\}_{j\leq T}=B$ for any given $B$. Since the set of $B$ is countable, this proves that there are at most countably many $\eta$, for which billiard trajectory of $(q,\eta)$ is a loop.



\end{proof}

\begin{rem}The Lemma
can also be extracted from the  articles \cite{St1,St2} of L. Stoyanov.
\end{rem}

\section{\label{QER} Boundary quantum ergodic restriction theorems }

In this section, we briefly  review the statement of the quantum ergodic restriction theorem to the boundary of \cite{HZ}.
Roughly speaking, the results says that if the  billiard ball map $\beta$ on $B^* \partial M$ is ergodic, then the boundary
values (Cauchy data) $u_j^{b}$ of eigenfunctions are quantum ergodic.
We define
\begin{equation}\label{gammadef}
\gamma(q) = \sqrt{1 - |\eta|^2}, \quad q = (y,\eta),
\end{equation}
and put
{\large
\begin{table}[h]
\begin{tabular}{|c|c|c|c|c|}  \hline
 B & $B \phi_{\lambda}$  &   $\phi_{\lambda}^{b}$ & $d\mu_B$   \\
\hline
 Dirichlet &  $u|_{Y}$ &  $\lambda^{-1} \partial_{\nu} \phi_{\lambda} |_{Y}$ & $\gamma(q) d\sigma $ \\
\hline
 Neumann &   $\partial_{\nu} \phi_{\lambda} |_{Y}$  &  $\phi_{\lambda}|_{Y}$  & $\gamma(q)^{-1} d\sigma$  \\
\hline

 \end{tabular}
 \caption{Boundary Values}\label{boundaryv}
\end{table}
 }
\bigskip

The third column consists of the non-zero part of the Cauchy data (Definition
\ref{CDDEF}) for eigenfunctions satisfying the associated boundary condition.
The measures in the right column are the so-called quantum limits of
the Cauchy data.

\begin{theorem}\label{main} \cite{HZ,ctz,Bu}
Let  $M $ be a compact  manifold with boundary and with  ergodic billiard map. Let $\{\phi_j^{b}\}$
be the boundary values of the eigenfunctions
$\{\phi_j\}$ of $\Delta_B$ on $L^2(M)$ in the sense of the table
above.
 Let $A_h$ be a semiclassical operator of order zero on $\partial M$. Then there is a subset $S$ of the positive integers, of density one,
such that
\begin{equation}\begin{gathered}
\lim_{j \to \infty, j \in S} \langle A_{h_j} \phi_j^b,
\phi_j^b \rangle = \omega_B(A),
\end{gathered}
\label{main-eqn}\end{equation}
where $h_j = \lambda_j^{-1}$ and $\omega_B$ is the classical state on the table above.
\end{theorem}

We only apply the result to `multiplication operators' by $f \in C(\partial M)$
in this article, in which case it takes the form \eqref{QES}.

The main idea of Theorem \ref{main} is that the Cauchy data of the eigenfunctions
provide a quantum analogue of a cross-section for the billiard flow, just
as the inward pointing unit tangent vectors $S^*_{\rm{in}, \partial M} M$
along the boundary provide a classical cross section. The interior quantum
ergodicity of eigenfunctions thus implies boundary quantum ergodicity.
In \cite{ctz}, this results was generalized to any smooth hypersurface of
$M$. For Theorem \ref{theo1a} we only need the case where the hypersurface
is the boundary.

\section{Restriction of the wave group of a Sinai billiard to the boundary}

In this section, we prepare for the proof of Theorem \ref{K}  by considering
 the wave front set of the Cauchy data
of the Dirichlet, resp. Neumann, wave kernel along the boundary. We also
recall the boundary local Weyl law.

The billiard flow arises in spectral problems because the singularities of the fundamental solution of the wave
equation propagate along billiard trajectories.  The billiard flow relevant to our problem is thus determined by
propagation of singularities for the wave equation on a non-positively curved surface \eqref{M} in the exterior
of a finite union of disjoint convex obstacles. In this case, the singularities which intersect the boundary are known as grazing rays. It was  proved independently by R. B. Melrose and M. E. Taylor  \cite{M,Tay}  that a singularity propagating along a grazing ray simply continues along the same
geodesic of the ambient space when it touches the boundary. Consequently, the dynamical billiard flow of the previous
section coincides with the propagation of singularities.

The Melrose-Taylor diffractive parametrix is used to determine the precise
singularities of solutions of wave equations near grazing rays, i.e. billiard trajectories
intersecting $\partial M$ tangentially.
In \cite{SoZ3}, the diffractive
parametrix of Melrose-Taylor is used to prove Theorem \ref{SoZ}. The
remaining properties of the wave kernel needed for Theorem \ref{theo1a}
are more elementary and do not require the diffractive parametrix.

\subsection{Wave front set of the wave group}

We denote by
\begin{equation} \label{EBSB} E_B(t) = \cos \big( t \sqrt{ -
\Delta_{B}}\;\big), \;\;\; \mbox{resp.}\;\;S_B(t) = \frac{\sin
\big(t \sqrt{- \Delta_{B }}\big)}{\sqrt{ -\Delta_{B }}}
\end{equation} the even (resp. odd) wave operators  $(M, g)$ with boundary conditions
$B$. The  wave group  $E_B(t)$  is the
solution operator of the  mixed problem
$$\left\{ \begin{array}{l} \bigl(\frac{\partial^2}{\partial t^2} - \Delta) E_B(t, x, y)=0, \\ \\ \quad E_B(0,x, y)=
\delta_x(y),\;\;\ \frac{\partial}{\partial t}  E_B(0,x, y)= 0,
\;\; x, y \in M;
\\ \\ B E_B(t, x, y) = 0,\;\; x \in \partial M \end{array} \right. $$

The wave front sets of $E_B(t, x, y)$ and $S_B(t, x, y)$ are determined by the    propagation of
singularities theorem of \cite{MSj}  for the mixed Cauchy Dirichlet (or Cauchy Neumann)
problem for the wave equation. We  from   \cite{HoI-IV} (Vol.
III, Theorem 23.1.4 and Vol. IV, Proposition 29.3.2) that
\begin{equation} \label{WFEB} WF (E_B(t, x, y)) \subset \bigcup_{\pm} \Lambda_{\pm},  \end{equation} where $\Lambda_{\pm} = \{(t, \tau,x, \xi, y, \eta):\; (x, \xi)= \Phi^t(y, \eta), \;\tau =\pm |\eta|_y\} \subset T^*(\mathbb R \times \Omega \times \Omega)$ is the graph of the generalized (broken) geodesic flow, i.e. the billiard flow $\Phi^t$. The same is true for $WF(S_B)$. As mentioned above, the
broken geodesics in the setting of \eqref{M} are simply the geodesics of the ambient negatively curved surface,
with the equal angle reflections at the boundary; tangential rays simply continue without change at the impact.

\subsection{Restriction of wave kernels to the boundary}

A key object in the proof of Theorem \ref{SoZ}   is the analysis of the  restriction of the Schwartz  kernel $E_B(t, x, y) $ of $\cos t \sqrt{\Delta_B}$
to $\mathbb{R} \times \partial M \times \partial M$ and further to
$\mathbb{R} \times \Delta_{\partial M \times \partial M}$, where $ \Delta_{\partial M \times \partial M}$
is  the diagonal  of $\partial M \times \partial M$.  We denote by $d q$ the surface measure on the boundary
$\partial M$, and by $r u =
u|_{\partial M}$ the trace operator. We denote by $E_B^b(t, q', q) \in \mathcal{D}'( \mathbb{R} \times \partial M
\times \partial M)$ the following boundary traces of the Schwartz kernel $E_B(t, x, y)$ defined in (\ref{EBSB}):

\begin{equation} \label{b}
E_B^b(t, q', q) =
\left\{\begin{array}{ll}
r_{q'} r_{q} \partial_{\nu_{q'}} \partial_{\nu_q} E_D(t, q', q) , & \;\; \;\;\; \mbox{Dirichlet}\\ \\
r_{q'} r_{q}\;  \; E_N(t, q', q), & \;\; \mbox{Neumann}
\end{array} \right. \end{equation}
The subscripts $q', q$ refer to the variable involved in the
differentiating or restricting.
Henceforth we use the notation $\gamma_q^B$ for the boundary trace. Thus, $\gamma_q^B = r_q$ in
the Neumann case and $\gamma_q^B = r_q \partial_{\nu_q}$ in the Dirichlet case.

The sup norm bounds of Theorem \ref{SoZ}   are derived in \cite{SoZ3} from an analysis of the singularities of the Cauchy data of the wave kernel on the diagonal of the boundary,
\begin{equation} \label{BTW} E_B^b(t, q, q) =  \sum_{j = 1}^{\infty}  \cos (t \lambda_j)
|\phi_j^b(q)|^2. \end{equation}   The Kuznecov
asymptotics of Theorem \ref{K}  are proved by studying the off-diagonal
integrals,
\begin{equation} \label{BTWK}
\int_{\partial M} \int_{\partial M} E_B^b(t, q, q')  ds(q) ds(q')=  \sum_{j = 1}^{\infty}  \cos (t \lambda_j)\left|\int_{\partial M}  \phi_j^b(q) ds(q) \right|^2.
\end{equation}
The contribution of the tangential directions to the singularities of \eqref{BTWK}
is negligible, and that is why the diffractive parametrix is not needed.

\subsection{\label{WF}Wave front set of the restricted wave kernel}

%

The first and simplest piece of information is the wave front set of \eqref{BTW}. It follows from \eqref{WFEB} and from
  standard results on pullbacks of wave front sets under maps, the  wave front set of  $E_B^b(t, q, q')$ consists of co-directions of broken trajectories which begin and end on $\partial M$.
That is,
\begin{equation} \label{WF1}
\begin{array}{lll}  WF (\gamma_q^B \gamma_{q'}^B E(t, q, q'))
&\subset & \{(t, \tau, q,
\eta, q', \eta') \in B^* \partial M \times B^* \partial M: \\&&\\&&[\Phi^t(q, \xi(q, \eta))]^T = (q', \eta'), \;
\tau = - |\xi| \}.
\end{array}
\end{equation}
Here, the superscript $T$ denotes the tangential projection to $B^* \partial M$.
 We refer to Section 2 of \cite{HeZ} for an extensive discussion.
It follows from \eqref{WF1} that  \begin{equation} \label{WF2} \begin{array}{lll} WF (\gamma_q^B \gamma_{q'}^B E(t, q, q))
& \subset &  \{(t, \tau, q, \eta, q, \eta') \in B^*_q \partial M \times B^*_q \partial M: \\&&\\&& [G^t(q, \xi(q,
\eta))]^T = (q, \eta'),\;\; \tau = - |\xi(q, \eta)|\}. \end{array}  \end{equation}
Thus, for $t \not= 0$, the  singularities of the boundary  trace $\gamma_q^B
\gamma_{q'}^B E(t, q, q)$ at $q \in \partial M$  to broken
bicharacteristic
 loops based at $q$ in
$\overline{M}$. When $t = 0$ all inward pointing co-directions belong to the wave front set.


\subsection{Local Weyl law}

The boundary local Weyl law gives an asymptotic formula for the spectral averages of the expected value of an observable $A_h$ relative to boundary traces of eigenfunctions. The relevant  algebra of observables in our setting as in \cite{HZ}
is the algebra  $\Psi_h^0(\partial M)$ of zeroth order
semiclassical pseudodifferential operators on $\partial M,$ depending on
the parameter $h \in [0, h_0]$. We denote the symbol of  $A = A_h
\in \Psi_h^0(\partial M)$ by  $a = a(y,\eta,h)$. Thus $a(y, \eta) = a(y,
\eta, 0)$ is a smooth function on $T^*\partial M$; we  may without loss of
generality assume it is compactly supported.
 We further define  states on the algebra $\Psi_h^0(\partial M)$ by
\begin{equation}\begin{aligned}
\omega_B(A) &= \frac{4}{\operatorname{vol}(S^{n-1})\operatorname{vol}(M)} \int_{B^*\partial M} a(y,\eta)  d\mu_B.
\end{aligned}\end{equation}
Here, as in Table \ref{boundaryv},

$$d\mu_B = \gamma(q) d\sigma \; (\mbox{Dirichlet}), \;\;\;d \mu_B =  \gamma(q)^{-1} d\sigma\;\; (\mbox{Neumann})
$$
where $$
\gamma(q) = \sqrt{1 - |\eta|^2}, \quad q = (y,\eta).$$

For either Dirichlet or Neumman boundary conditions,  the local Weyl law
is proved in Lemma 1.2 of  \cite{HZ}:

\begin{proposition} \label{LWLHZ}  Let $A_h$ be  a zeroth order semiclassical operator on $\partial M$. Then,
\begin{equation}
\lim_{\lambda \to \infty} \frac1{N(\lambda)} \sum_{\lambda_j \leq \lambda} \langle A_{h_j} \phi_j^b, \phi_j^b \rangle \to \omega_{B}(A)\label{Weyl-D}\end{equation}
\end{proposition}

Note that in \cite{HZ} the kernel of $A_h$ was assumed to be disjoint from the singular set of $\partial M$,
but in this article the singular set is empty.





\section{\label{KUZSECT} Kuznecov sum formula for the boundary integral: Proof of Theorem \ref{K}}

The general Kuznecov formula in \cite{z} for $C^{\infty}$ Riemannian manifolds $(M, g)$ without
boundary is a singularity expansion for the distribution
\begin{equation} \label{SH} S_H(t) = \int_H \int_H E(t, q, q') ds(q') ds(q), \end{equation}
where $H \subset M$ is a smooth submanifold and where
$$E(t) = \cos t \sqrt{\Delta} $$
is the even wave kernel.
The singularities of $S_H(t)$ in the boundaryless case were shown to correspond
to trajectories of the geodesic flow which intersect $H$ orthogonally at two distinct times, and to be singular
at the difference $T$ of these times. We refer to such trajectories  as H-{\it  orthogonal} geodesics.

Theorem \ref{K} is a generalization of the  Kuznecov formula  of \cite{z} to boundary traces on  surfaces
with concave boundary. We do not consider the full singularity expansion as in \cite{z} but
only the singularity at $t = 0$ of
\begin{equation}\label{Sft} \begin{array}{lll} S_f(t): & = & \int_{\partial M} \int_{\partial M}  E_B^b(t, q, q') f(q) f(q') ds(q) ds(q')
\\ &&\\ & = &  \sum_{j} \cos t \lambda_j \left|\int_{\partial M} f(q)  \phi^b_j(q) ds(q) \right|^2 .
\end{array} \end{equation}
where $E_B(t) = \cos t \sqrt{\Delta_B} $
is the even wave kernel with either Dirichlet or Neumann boundary conditions.

 Theorem \ref{K} is   a corollary  of Theorems 1, Proposition 2 and Theorem 3  of  \cite{HHHZ}, which are proved for general
manifolds with boundary.

\begin{theorem}\label{thm:BFSS}
Let $\rho\in\mathcal S(\mathbb{R})$ be such that $\hat\rho$ is identically $1$ near $0$, and has sufficiently small support.  Let  $f\in C^\infty(\partial M)$. Then for either  the Dirichlet or Neumann boundary conditions,
\begin{equation}
f(x)=\lim_{\lambda\to\infty} \frac{\pi}{2}\sum_j\rho(\lambda-\lambda_j)\langle\phi_j^b,f\rangle\phi_j^b(x),
\label{eq:BFSS}\end{equation}
where $\langle\cdot,\cdot\rangle=\langle\cdot,\cdot\rangle_{\partial M}$ denotes the inner product in $L^2(\partial M)$.
\end{theorem}


Evidently, Theorem \ref{K} follows by taking the inner product with $f$ on both sides of the equation.
We now give  a relatively self-contained proof, different from that of \cite{HHHZ}, which  exploits   the concavity of the boundary. For purposes of this
article it is only necessary to prove  Corollary \ref{cor81}. We only use \cite{HHHZ} to calculate
one constant at the end.

\subsection{Sketch of the proof}


The first step is the following

\begin{lemma} \label{ep}   There exists $\epsilon_0 > 0$ so that the
$$\mbox{sing supp} S_f(t) \cap (- \epsilon_0, \epsilon_0) = \{0\}. $$
\end{lemma}

\begin{proof}

By \eqref{WF2} and standard pullback and pushforward calculations for wave front sets as in \cite{z},  the singular
support of $S_f(t)$ consists of $t = 0$ together with  the `sojourn times' equal to lengths of billiard trajectories  which hit the boundary
 orthogonally   at both endpoints. Such a billiard trajectory either (i)  intersects two distinct  components of $\partial M$,
in which case its length is bounded below by the minimum  distance $d_{\mathcal{O}}$ between the components, or (ii) intersects the same component
orthogonally. However if it starts off orthogonally to the boundary, it cannot intersect the boundary again until
it departs from a Fermi normal coordinate chart along the boundary, i.e. the radius of the maximal tube around
each component which is embedded in $M$.  The minimum $\epsilon(M, g)$  over components of the maximal embedding
radius gives a geometric lower bound for
its length in this case.

Thus, we may let $\epsilon_0 = \min\{\epsilon(M, g), d_{\mathcal{O}}\}$.

\end{proof}

To prove Theorem \ref{K}  it thus suffices to determine the singularity at $t = 0$ of $S_f(t)$. Equivalently, we prove a  smoothed version and then use a cosine Tauberian theorem. As mentioned in the introduction, we only need a sufficiently accurate asymptotic
expansion and remainder  to prove Corollary \ref{cor81}.

To study the singularity at $t = 0$, we introduce a  smooth cutoff  $\rho \in \mathcal{S}(\mathbb{R})$ with $\mbox{supp} \hat{\rho} \subset (-\epsilon, \epsilon)$,
where $\hat{\rho}$ is the Fourier transform of $\rho$ and $\epsilon < \epsilon_0$.
With no loss of generality we assume that $\hat{\rho} \in C_0^{\infty} (\mathbb{R})$ is a positive even  function   such that $\hat\rho$ is identically $1$ near $0$, has support in $[-1,1]$ and is decreasing on $\mathbb{R}_+$.  We  then study
\begin{equation} \label{Sfeq} S_f(\lambda, \rho) = \int_{\mathbb{R}} \hat{\rho} (t) \;S_f(t) e^{i t \lambda} dt. \end{equation}
Our purpose is to obtain an asymptotic expansion of $S_f(\lambda, \rho)$ as $\lambda \to \infty$.

\begin{proposition} \label{SF} $S_f(\lambda, \rho)$ is a semi-classical Lagrangian distribution whose asymptotic expansion in
both the Dirichlet and Neumann  cases
is given by

\begin{equation}\label{eq:BFSSb}
S_f(\lambda, \rho) =  \frac{\pi}{2}\sum_j (\rho(\lambda-\lambda_j) + \rho(\lambda+\lambda_j)) |\langle\phi_j^b,f\rangle|^2
= ||f||_{L^2(\partial M)}^2 + o(1),
\end{equation}
\end{proposition}

\begin{proof}  For $\epsilon < \epsilon_0$ in Lemma \ref{ep},   we only need to determine the
contribution of the main singularity of $S_f(t)$ \eqref{Sft}  at $t = 0$.  As in \cite{SoZ}, the $\rho(\lambda + \lambda_j)$ term contributes $\mathcal{O}(\lambda^{-M})$
for all $M > 0$ and therefore may be neglected.

 To show that $S_f(t)$ is a Lagrangian distribution and to determine its singularity at $t = 0$,  it suffices to construct a sufficiently precise parametrix for the non-tangential cutoff $\chi(q, D_t, D_q) E^b_B(t, q,q')$ of the wave kernel for
small $t$ in some neighborhood of the diagonal in $\partial M \times \partial M$. 
In the case of a concave boundary, for small times $E^b_B(t, q, q')$
is singular only when $t = 0$ and $q = q' \in \partial M$. This follows from \eqref{WF1} and the fact that there
do not exist any broken geodesic billiard trajectories from $q$ to $q'$ for small $t$ except when $t= 0, q = q'$.
Thus it suffices to determine the singularity at $t = 0$.

We introduce  a pseudo-differential cutoff
$\chi(q, D_t, D_q)$ on $\mathbb{R} \times \partial M$ whose
symbol vanishes in an arbitrarily small $\delta$-neighborhood of the tangential directions to $\partial M$. More precisely,
as in \cite{HHHZ}, we let $\chi(y, D_t, D_y)$ be a pseudodifferential operator on $\mathbb{R} \times \partial M $ with symbol of the form
\begin{equation}
\chi(y, \tau, \eta) = \zeta( |\eta|_{\tilde g}^2/\tau^2) (1 - \phi(\eta, \tau)),
\label{zeta}\end{equation}
where $\zeta(s)$ is supported where $s \leq 1 - \delta$ for some positive $\delta$, and $\phi \in C_c^\infty(\mathbb{R}^n)$ is equal to 1 near the origin.

We then decompose
  $ E_B^b(t, q, q') $ into an almost tangential part and a part with empty wave front set in tangential directions,
$$ E_B^b(t, q, q')  = (I - \chi(q, D_t, D_q)) E_B^b(t, q, q') + \chi(q, D_t, D_q)) E_B^b(t, q, q') . $$ We then have
corresponding terms in $S_f(t)$,
$$S_f^{\epsilon}(t)  = \int_{\partial M} \int_{\partial M}  (I - \chi(q, D_t, D_q)) E_B^b(t, q, q')  f(q) f(q') dS(q) dS(q') $$
and
$$S_f^{> \epsilon}(t) =  \int_{\partial M} \int_{\partial M}  \chi(q, D_t, D_q) E_B^b(t, q, q')  f(q) f(q') dS(q) dS(q'). $$

As in \cite{z} (1.6) we express $S_f(t)$ and $S_f(\lambda, \rho)$ in terms of pushforward under the submersion
$$\pi: \mathbb{R} \times \partial M \times \partial M \to \mathbb{R}, \;\;\; \pi(t, q, q') = t. $$
From \eqref{WF1} we find that for $t \in (-\epsilon, \epsilon) $,
\begin{align*}
WF&(S_f^{\epsilon}(t)  ) =\\
&\{(0, \tau): \pi^* (0, \tau) = (0, \tau, 0, 0) \in WF  (I - \chi(q, D_t, D_q)) E_B^b(t, q, q') \} .
\end{align*}
These wave front elements correspond to the points $(0, \tau, \tau  \nu_q, \tau  \nu_q) \in T^*_0 \mathbb{R}
\times T^*_{q, in} M \times T^*_{q, in} M$, i.e. where both covectors are co-normal to $\partial M$. Indeed,
as in (1.6) of \cite{z} the wave front set of $S_f(t)$ is the set
$$\{(t, \tau) \in T^* \mathbb{R}: \exists (x, \xi, y, \eta) \in C'_t \cap N^*(\partial \Omega) \times N^* \partial \Omega\} $$
in the support of the symbol.
However, due to the tangential cutoff  $( (I - \chi(q, D_t, D_q)) E_B^b(t, q, q')  $ has no such co-normal vectors
in its wave front set. Thus, we may neglect the tangential part of $E^b_B(t, q, q')$ in determining the asymptotics of $S_f(\rho, \lambda)$. But then it follows from \cite{M,Tay} that the non-tangential part has a geometric optics Fourier
integral representation, i.e. $S_f(t)$ is classical co-normal at $t = 0$.

The non-tangential part $\chi(q, D_t, D_q) E_B^b(t, q, q') $ may be expressed in terms of the ``free wave kernel"
or ambient wave kernel $E_X(t,x, y)$ of $(X, g)$.
 Given $q \in \partial M$,
we may separate $M$ into an illuminated region (from a source at $q$) and a shadow region.
The hyperplane $T_{q} \partial M \subset T_{q} X$ divides the full tangent space into two halfspaces. By concavity, geodesics with
initial direction $\xi$ in the lower half-space lie in $M$  for $|t| < \epsilon$. We call the image of the unit tangent
vectors in the lower half space under the geodesic flow up to time $\epsilon$ the `illuminated region' for a
point source at $q$.  The complement
of the illuminated region in $T_{\epsilon} M \cap M$ is the `shadow region'.   Geodesics with $\xi$ in
the upper half plane exponentiate to $X \backslash M$ for at least a short time.
 If we cutout all $\xi$ whose
angle to $T_{q} \partial M$ is $\leq \delta$, then geodesics in the upper half space remain in $X \backslash M$
for a uniform length of time, which may assume with loss of generality is $> \epsilon$. We write
this set as $T_{ \delta} X$.


We then cut off the ambient cosine wave kernel $ E_X(t, x, q) $  in the shadow region, removing  the singularities of the ambient kernel due to geodesics that leave $M$ and travel through $X \backslash M$.

\begin{lemma} \label{lem8}
For $|t| < \epsilon$, the non-tangential part $ \chi(q, D_t, D_q)) E_B^b(t, q, q')$ of $E_B^b(t, q, q')$
can be expressed as $A(q, D_q) \gamma^B_q \gamma^B_{q'} E_X(t, q, q')  \chi(q, D_t, D_q)) $ where $E_X(t, x, y)$ is the cosine
wave kernel of $(X, g)$ and $A $ is a pseudo-differential operator on $\partial M$ of order zero.
\end{lemma}


Indeed, by the wave front calculations of \S \ref{WF},  $$E_B^b(t, x, q)  \chi(q, D_t, D_q), \;\; \mbox{resp}. \;\; E_X^b(t, x, q)  \chi(q, D_t, D_q)$$ are Fourier integral operators with
the same wave front set equal to a pullback of the diagonal at $t = 0$,  and therefore by
the calculus of Fourier integral operators, there  exists a pseudo-differential operator $A$ whose composition with the cutoff free wave kernel agrees to any given order with the $E_B^b$ kernel.

Therefore, to prove Proposition \ref{SF} it is  sufficient to consider the integral
\begin{equation} \label{INT} \begin{array}{l} \int_{\mathbb{R}} \int_{(q,q,') : d(q,q') < \epsilon} \hat{\rho}(t) \chi_{q}(q, D_t, D_q) e^{i t \lambda} \\ \\  \gamma_B^q \gamma_B^{q'} \ E_X(t, x, q) f(q) f(q') dS(q) dS(q')dt. \end{array} \end{equation}


We  use a Hormander style small time  parametrix for $E_X(t, x, q)$, i.e. there exists an amplitude $A$ so that modulo
smoothing operators,
$$\gamma_B^bE_X(t, x, q) = \int_{T_q^* X} A(t, x, q, \xi) \exp \left(i (\langle Exp_q^{-1}(x), \xi \rangle - t |\xi|)\right) d \xi.$$
The amplitude has order zero.
We then take the boundary trace and apply the  cutoff operator $\chi(q, D_t, D_q)$, which modifies
the amplitude of \eqref{INT}  as a sum of terms with the same support as $  \chi_q(\tau, \xi) $.

Changing variables $\xi \to \lambda \xi$, \eqref{INT}  may be expressed in the form,
$$ \begin{array}{l} \lambda^2 \int_{\mathbb{R}} \int_{\partial M} \int_{\partial M} \int_{T_q^* X} \hat{\rho}(t) e^{it  \lambda} \chi_q(\xi) \\ \\A(t, q', q, \lambda \xi) \exp( i \lambda \left(\langle Exp_q^{-1}(q'), \xi \rangle - t |\xi|\right)) f(q) f(q;)  d \xi dt dS(q) dS(q'). \end{array}$$
 We now compute
the asymptotics by the stationary phase method.

We already know that for small $t$, the  phase
$$ t + \langle Exp_q^{-1}(q'), \xi \rangle - t |\xi| $$
 is stationary only at $t = 0$ and $q = q'$. We calculate the expansion by putting
the integral over $T^*_q X $ in polar coordinates,
$$\begin{array}{l}  \int_{\mathbb{R}} \int_0^{\infty} \int_{\partial M} \int_{\partial M} \int_{S_q^* X} \hat{\rho}(t) e^{it  \lambda} \chi_q(\xi) A(t, q', q) \\ \\\exp( i \lambda \rho \left(\langle Exp_q^{-1}(q'), \omega \rangle - t  \right)) f(q) f(q;)  \rho^{n-1} d \rho d \omega dt dS(q) dS(q'), \end{array} $$
and in these coordinates the phase becomes,
$$ \Psi(q, \rho, t, \omega, q'): =  t +  \rho \langle Exp_q^{-1}(q'), \omega \rangle - t \rho.$$
We get a non-degenerate critical point  in the variables $(t, \rho)$ when $\rho =1, t =  \langle Exp_q^{-1}(q'), \omega \rangle.$
Eliminating these variables  by stationary phase , we get
\[
 \frac{1}{\lambda} \int_{\partial M} \int_{\partial M} \int_{S_q^* X}  e^{i  \lambda   \langle Exp_q^{-1}(q'), \omega \rangle} \chi_q(\omega) \tilde{A}(t, q', q, \rho \omega)  f(q) f(q')   d \omega  dS(q) dS(q'),
\]
for another amplitude $\tilde{A}$. Now the phase is
$$\Psi_q(q', \omega) = \langle Exp_q^{-1}(q'), \omega \rangle. $$
We fix $q$ and view the phase as a function of $(q', \omega) \in \partial M \times S^*_{q, \rm{in}} X $ (the inward pointing unit tangent vectors;  due to the cutoff, the amplitude vanishes on the outward pointing vectors).  We view
$S^*_{q, \rm{in}}$ as the lower hemisphere $S_-^{n-1}$
of the sphere $S^{n-1} = S^*_q X$. Note that the integral is compactly supported
in the interior of the hemisphere, so critical points on the boundary are irrelevant. We  claim that the phase has a critical point
if and only if $q'= q$ and $\omega \bot T_q \partial M$, i.e. $\omega =  \nu_q$ ($\nu_q$ being
the inward unit normal).  Moreover, the stationary phase point is
non-degenerate.
Since we are working  close to the diagonal, we use geodesic normal coordinates
$Exp_q^{-1}(q')  = x$ and consider the inverse image  $Y = Exp_q^{-1}(\partial \Omega)$ of a small piece of $\partial \Omega$ near $q$ in $T_q X$.  Let
 $q(y)$ be  a local parametrization of $Y$ with $\partial_j q(y)|_{y = 0}$ an
orthonormal frame of $T_q \partial \Omega$.
Then in any dimension $n$  we may write the phase in local coordinates as
$$\Psi_q(y, \omega) = \langle q(y), \omega \rangle \;\; y \in \R^{n-1}, \omega
\in S^*_y X. $$
In this parametrization, $q(0) = q$ and $Y$ is locally the graph of a convex
function over $T_q \partial \Omega$.
As is well-known,
\begin{equation} \label{C1} \nabla_{y} \langle q(y), \omega \rangle = \omega^T,
\end{equation}
the tangential projection of $\omega$ to $T_{q(y)} \partial \Omega$  and as
above we find that $\omega \bot T_q \partial M$
at the stationary phase point.  Similarly,
\begin{equation} \label{C2} \nabla_{\omega} \langle q(y), \omega \rangle = q(y)^T, \end{equation}
the tangential projection of $q(y)$ to $T_{\omega} S^{n-1}$.
Then \eqref{C1}-\eqref{C2} occur for $q(y)$ near $q$ if and only if (i)
$q(y) = 0$ (in normal coordinates)   and $\omega = \nu_q$, or (ii)  $\frac{q(y)}{|q(y)|} = \pm \omega = \pm \nu_q$. In the first case,  $q' = q$;
the second case does not occur for a  concave hypersurface.
The Hessian in $(y, \omega) \in \partial M \times S_-^{n-1}$ at the critical point
has the form,
\[
\begin{pmatrix}
-II_q &  I_q \\ &&\\   I_q & 0
\end{pmatrix}
\]

 where $II_q$ is the second fundamental form at $q$ with respect to
$\nu_q$ (see e.g. \cite{HoI-IV} \S 7.7  for the calculation of the upper left block).
The lower right block is zero since $q(y) = 0$ at the critical point. For
the off-diagonal blocks, we implicitly identify $T_q \partial \Omega$
with $T_{\nu_q} S^*_q X$ since $\omega = \nu_q$ at the critical point.
We parametrize $\omega = \omega(\theta) \in S^{n-1}_-$ so that $\partial_j \omega = e_j$ at $\nu_q$, the same frame we use for $T_q \partial \Omega.$ The determinant of the Hessian
is evidently of modulus one.
 In the two dimensional case, we obtain
another factor of $\lambda^{-1}$ from the stationary phase expansion. Therefore  \eqref{INT}
is asymptotic to a multiple of
$$\int_{\partial M} f^2(q) d S(q),$$
as stated in Proposition \ref{eq:BFSSb}.

To determine the multiple, or more precisely to show that it is positive, we need to find the principal symbol
of the pseudo-differential operator in \eqref{lem8}. In fact, it is a constant equal to $2$ when pulled back to
$\partial M$.  This follows from the calculations in   \cite{HeZ} Proposition 4 and in \cite{HHHZ}, which prove:

\begin{lemma}\label{HZ-result}
Suppose that $\hat \rho$ is supported in $[-\epsilon, \epsilon]$ and equal to $1$ in a neighbourhood of $0$.Then, for sufficiently small $\epsilon$ (depending on $\delta$),
\begin{enumerate}
\item the kernels of
\begin{align*}
\hat \rho(t) \chi(y, D_t, D_y) \circ
\gamma_q^b \gamma_{q'}^b\cos(t\sqrt{\Delta_D}),\\
\quad\hat \rho(t)
\gamma^b_q \gamma_{q'}^b \cos(t\sqrt{\Delta_D}) \circ \chi(y, D_t, D_y)
\end{align*}
are  distributions conormal to $\{ y = y' , t = 0 \}$ with principal symbol
\begin{equation}
2\chi(y, \tau, \eta) \left( 1 - \frac{|\eta|_{\tilde g}^2}{\tau^2} \right)^{\frac12};
\label{wavesymbolD}\end{equation}

\item the kernels of
\begin{align*}
\hat \rho(t) \chi(y, D_t, D_y) \circ
\gamma_q^b \gamma_{q'}^b \cos(t\sqrt{\Delta_N}),\\
\quad\hat \rho(t)
\gamma_q^b \gamma_{q'}^b \cos(t\sqrt{\Delta_N}) \circ \chi(y, D_t, D_y)
\end{align*}
are distributions conormal to $\{ y = y' , t = 0 \}$ with principal symbol
\begin{equation}
2\chi(y, \tau, \eta) \left( 1 - \frac{|\eta|_{\tilde g}^2}{\tau^2} \right)^{-\frac12} .
\label{wavesymbolN}\end{equation}
\end{enumerate}
\end{lemma}

This completes the proof of Theorem \ref{K}.
\end{proof}

\begin{rem}

Above, we use the notation $\cos (t \sqrt{\Delta_D})$ resp. $\cos (t \sqrt{\Delta_N})$ in place of $E_B(t)$ since the formula is different
in the Dirichlet, resp. Neumann cases. Also, the symbols are
homogeneous analogues of \eqref{gammadef}.
\end{rem}

\section{Proof of Theorem \ref{theo1a}}
In this section, we give a proof of Proposition \ref{PROP} and Theorem \ref{theo1a} for Neumann eigenfunctions. The argument for Dirichlet eigenfunctions is exactly the same.
\subsection{Proof of Proposition \ref{PROP}}
Firstly let $\beta \subset \partial M$ be an interval and let $f \in C_0^\infty(\partial M)$ be a function such that
\begin{align*}
f(x) &\geq 0\quad x \in \partial M\\
f(x) &= 0 \quad x \notin \beta\\
f(x) &>0 \quad x \in \beta
\end{align*}

Denote by $N(\lambda)$ the number of eigenfunctions in $\{j~|~\lambda<\lambda_j<2\lambda \}$. We have by Theorem \ref{K} and Chebyshev's inequality,
\[
\frac{1}{N(\lambda)}\left|\left\{j~|~\lambda<\lambda_j<2\lambda,~\left|\int_{\gamma_i} f \phi_j ds\right|^2 \geq \lambda_j^{-1}M \right\}\right| = O_f(\frac{1}{M}).
\]
Corollary \ref{cor81} follows immediately.

Note that
\[
\int_{\partial M} f|\phi_j|^2 ds \leq \int_{\partial M} f|\phi_j| ds \sup_{x \in \partial M} |\phi_j(x)|.
\]
For a density $1$ subsequence $\{\phi_j\}_{j \in A}$ which satisfies \eqref{QES}, we have
\[
\int_{\partial M} f|\phi_j|^2 ds \gg_f 1.
\]
Therefore from the third assumption in Theorem \ref{theo1},
\[
\int_{\partial M} f|\phi_j| ds > 2M\lambda_j^{-\frac{1}{2}}
\]
is satisfied for all sufficiently large $j \in A$. Combining with Corollary \ref{cor81}, this proves the existence of a subsequence of density $\geq 1-\frac{c}{M}$ which satisfies
\[
\int_{\partial M} f|\phi_j| ds > \left|\int_{\partial M} f \phi_j ds  \right|.
\]
Putting
\[
A_\beta = \left\{j~:~ \int_{\partial M} f|\phi_j| ds > \left|\int_{\partial M} f \phi_j ds  \right|\right\},
\]
we obtain
\[
\liminf_{N \to \infty}\frac{1}{N} |\{j<N~:~ j \in A_\beta\}| \geq 1-\frac{c}{M}.
\]
Since the left quantity does not depend on $M$, this proves Proposition \ref{PROP} for Neumann eigenfunctions.
\subsection{Proof of Theorem \ref{theo1a}}
Note that because $f$ is positive on $\beta$, a function $\phi_j$ has a sign change on $\beta$ if and only if $j \in A_\beta$. Let $R \in \mathbb{N}$ be fixed, and let $\beta_1, \cdots, \beta_R \subset \partial M$ be disjoint segments in $\partial M$. Then by Proposition \ref{PROP}, each $A_{\beta_k}$ ($1\leq k \leq R$) is a natural density $1$ subset of $\mathbb{N}$. Therefore $A(R)= \cap_{k=1}^R A_{\beta_k}$ is a density $1$ subset of $\mathbb{N}$, and any $\phi_j$ with $j \in A(R)$ has at least $R$ sign changes along $\partial M$. We apply the following lemma to conclude Theorem \ref{theoS} for Neumann eigenfunctions.
\begin{lemma}\label{lem2}
Let $a_n$ be a sequence of real numbers such that for any fixed $R>0$, $a_n>R$ is satisfied for almost all $n$. Then there exists a density $1$ subsequence $\{a_n\}_{n\in A}$ such that
\[
\lim_{\substack{n\to \infty \\ n \in A} }a_n = +\infty.
\]
\end{lemma}
\begin{proof}
Let $n_k$ be the least number such that for any $n \geq n_k$,
\[
\frac{1}{n}|\{j \leq n~|~a_j>k \}| > 1- \frac{1}{2^k}.
\]
Note that $n_k$ is nondecreasing, and $\lim_{k\to \infty}n_k = +\infty$.

Define $A_k \subset \mathbb{N}$ by
\[
A_k = \{n_k \leq j < n_{k+1}~|~ a_j>k\}.
\]
Then for any $n_k\leq m <n_{k+1}$,
\[
\{j\leq m~|~a_j>k\} \subset \bigcup_{l=1}^k A_l \cap [1,m];
\]
which implies by the choice of $n_k$ that
\[
\frac{1}{m}\left|\bigcup_{l=1}^k A_l \cap [1,m]\right| >1- \frac{1}{2^k}.
\]
This proves
\[
A=\bigcup_{k=1}^\infty A_k
\]
is a density $1$ subset of $\mathbb{N}$, and by the construction we have
\[
\lim_{\substack{n\to \infty \\ n \in A} }a_n = +\infty.
\]
\end{proof}

\subsection{Topological arguments: Euler inequality}

As done in \cite{JZ}, we can give a graph structure (i.e. the structure of a one-dimensional CW complex)
 to $Z_{\phi_{\lambda}}$ as follows.
\begin{enumerate}
\item For each embedded circle which does not intersect $\gamma$, we add a vertex.
\item Each singular point is a vertex.
\item Each intersection point in $\partial M \cap \left(\overline{Z_{\phi_\lambda}\backslash \partial M}\right)$ is a vertex.
\item Edges are the arcs of $Z_{\phi_\lambda} \cup \partial M$ which join the vertices listed above.
\end{enumerate}

This way, we obtain a graph embedded into the surface $M$.  We recall that an embedded graph $G$ in a surface
$M$ is a finite set $V(G)$ of vertices and a finite set $E(G)$ of edges which are simple (non-self-intersecting)
curves in $M$ such that any two distinct edges have at most one endpoint and no interior points in common.
The {\it faces} $f$ of $G$ are the  connected components of $M \backslash V(G) \cup \bigcup_{e \in E(G)} e$.
The set of faces is denoted $F(G)$. An edge $e \in E(G)$ is {\it incident} to $f$ if the boundary of $f$ contains
an interior point of $e$. Every edge is incident to at least one and to at most two faces; if $e$ is incident
to $f$ then $e \subset \partial f$. The faces are not assumed to be cells and the sets $V(G), E(G), F(G)$ are
not assumed to form a CW complex. Indeed the faces of the nodal graph of eigenfunctions are nodal domains, which do not have to be simply connected.

\begin{remark}\label{rem}
Every vertex has degree $\geq 2$, since any interior singular point is locally an intersection of simple curves,
as follows from the local Bers expansion around a zero \cite{ch}.
\end{remark}

Let $\iota: M \hookrightarrow \tilde{M}$ be an embedding into a closed surface. We assume that $\tilde{M} \backslash \iota \left(M\right)$ is a disjoint union of disks, and we denote by $h_M$ the number of connected components of $\tilde{M} \backslash \iota \left(M\right)$. (Such pair $\iota, \tilde{M}$ can be constructed, for example, by mapping cone.)

Let $v(\phi_\lambda)$ (resp. $\tilde{v}(\phi_\lambda)$) be the number of vertices, $e(\phi_\lambda)$ (resp. $\tilde{e}(\phi_\lambda)$) be the number of edges, $f(\phi_\lambda)$ (resp. $\tilde{f}(\phi_\lambda)$) be the number of faces, and $m(\phi_\lambda)$ (resp. $\tilde{m}(\phi_\lambda)$) be the number of connected components of the graph $G$ inside $M$ (resp. $\iota(G)$ inside $\tilde{M}$).

Then we have
\begin{align*}
v(\phi_\lambda) &= \tilde{v}(\phi_\lambda)\\
e(\phi_\lambda) &= \tilde{e}(\phi_\lambda)\\
m(\phi_\lambda) &= \tilde{m}(\phi_\lambda)
\end{align*}
while
\[
f(\phi_\lambda)+h_M = \tilde{f}(\phi_\lambda).
\]

Now by Euler's formula (Appendix F, \cite{g}),
\begin{align}\label{euler}
&v(\phi_\lambda)-e(\phi_\lambda)+f(\phi_\lambda)-m(\phi_\lambda) +h_M\\
=&\tilde{v}(\phi_\lambda)-\tilde{e}(\phi_\lambda)+\tilde{f}(\phi_\lambda)-\tilde{m}(\phi_\lambda) \geq 1- 2 g_{\tilde{M}}
\end{align}
where $g_{\tilde{M}}$ is the genus of the surface $\tilde{M}$.
\begin{theorem}\label{topology}
Let
\[
n(\phi_j)=\left\{ \begin{array}{lr}  \#Z_{\phi_j}\cap \partial M & \text{(Neumann case)}
\\
\#\Sigma_{\phi_j}\cap \partial M & \text{(Dirichlet case)} \end{array} \right.
\]
Then we have:
\[
N(\phi_j) \geq \frac{1}{2}n(\phi_j) +2 - 2g_{\tilde{M}}-h_M.
\]
\end{theorem}
\begin{proof}
Since faces of $G$ on $M$ are nodal domains of $\phi_j$, $f(\phi_j)=N(\phi_j)$. Observe that, in Neumann case, points in $Z_{\phi_j}\cap \partial M$ ($\Sigma_{\phi_j}\cap \partial M$, in Dirichlet case) correspond to vertices having degree at least $3$ on the graph. Also, every vertex has degree $\geq 2$ (Remark \ref{rem}). Therefore,
\begin{align*}
0&= \sum_{x:vertices} \mathrm{deg}(x) -2e(\phi_j) \\
&\geq 2\left(v(\phi_j)-n(\phi_j)\right)+3 n(\phi_j)-2e(\phi_j),
\end{align*}
i.e.
\[
e(\phi_j)-v(\phi_j) \geq \frac{1}{2}n(\phi_j).
\]
Plugging into \eqref{euler} with $m(\phi_j)\geq 1$, we obtain
\[
N(\phi_j) \geq \frac{1}{2}n(\phi_j) +2 - 2g_{\tilde{M}}-h_M.
\]
\end{proof}

 Theorem \ref{theo1a}   is an immediate consequence of Theorem \ref{theoS} and the topological argument, Theorem \ref{topology}.

\appendix
\section{Appendix on Density one}

Define the natural density of a set $A \in \mathbb{N}$ by
\[
\lim_{X\to \infty } \frac{1}{X}|\{x\in A~|~ x<X\}|
\]
whenever the limit exists. We say ``almost all" when corresponding set $A \in \mathbb{N}$ has the natural density $1$. Note that intersection of finitely many density $1$ set is a density $1$ set.
When the limit does not exist we refer to the $\limsup$ as the upper density and the
$\liminf$ as the lower density.

\subsection{Diagonal argument}

Let $\{f_n\} \subset C^{\infty}(H)$ be a countable dense subset of $C^0(H)$  with respect to the sup norm.
For each n, we have a family  $\Lambda_n(\lambda)$ of subsets for which
$$\frac{1}{N(\lambda)} \# \Lambda_n(\lambda) \to 1 $$
and such that
\begin{equation} \label{DENONE} \left\{ \begin{array}{l}\int_H f \phi_j^2 ds \to \int_H f d\nu, \; \mbox{as}\; \lambda \to \infty\; \mbox{with} \;\; \lambda_j \in \Lambda_n(\lambda) \\  \\
\lambda_j^{-1/2} \int_H f \phi_j ds \to 0. \end{array} \right.
\end{equation}

We may assume that $\Lambda_{n +1}(\lambda) \subset \Lambda_n(\lambda). $
For each $n$ let $\Lambda_n$ be large enough so that
$$\frac{1}{N(\lambda)} \# \Lambda_n(\lambda)  \geq 1 - \frac{1}{n}, \;\; \lambda \geq \Lambda_n.$$
Define
$$\Lambda_{\infty} (\lambda) : \Lambda_n(\lambda), \;\; \Lambda_{k} \leq \lambda \leq \Lambda_{k + 1}. $$
Then

$$\frac{1}{N(\lambda)} \# \Lambda_{\infty}(\lambda)  \geq 1 - \frac{1}{n}, \;\; \lambda \geq \Lambda_n, $$
so $D^*(\Lambda_{\infty}) = D_*(\Lambda_{\infty}) = 1$ and
and \eqref{DENONE} is valid for the sequence $\Lambda_{\infty}$.

\bibliography{bibfile}
\bibliographystyle{alpha}

\end{document}